\documentclass[a4paper, 12pt]{amsart}

\usepackage{amssymb}
\usepackage{amsthm}
\usepackage{amsmath}
\usepackage[mathscr]{eucal}
\usepackage{comment}
\usepackage{graphicx}
\usepackage{stmaryrd}

\theoremstyle{plain}
\newtheorem{thm}{Theorem}[section]
\newtheorem{prop}[thm]{Proposition}
\newtheorem{cor}[thm]{Corollary}
\newtheorem{lem}[thm]{Lemma}

\theoremstyle{definition}
\newtheorem{df}{Definition}[section]

\theoremstyle{remark}
\newtheorem{rmk}{Remark}[section]
\newtheorem*{ac}{Acknowledgements}

\newcommand{\zz}{\mathbb{Z}}

\newcommand{\rr}{\mathbb{R}}

\newcommand{\grsp}{\mathscr{M}}
\DeclareMathOperator{\grdis}{\mathcal{GH}}

\DeclareMathOperator{\hdis}{\mathcal{HD}}

\newcommand{\deltasecond}{H^{\times}}

\newcommand{\sqsq}[1]{\mathbf{#1}}

\newcommand{\disdis}{R}
\newcommand{\ctree}{\Upsilon}

\newcommand{\grsptwo}{\mathscr{PM}}

\newcommand{\qqcube}{\mathbf{C}}

\DeclareMathOperator{\pgrdis}{\mathcal{GH}_{*}}

\newcommand{\topemb}{\rho}

\newcommand{\treeset}{\mathscr{T}}

\DeclareMathOperator{\mea}{\deg}

\newcommand{\lbra}{\left(}
\newcommand{\rbra}{\right)}

\newcommand{\lbran}{\left\{}
\newcommand{\rbran}{\right\}}

\newcommand{\zantei}{\circ}

\newcommand{\intvl}{\mathbb{I}}

\newcommand{\sansan}[1]{\mathcal{Y}_{3}(#1)}

\newcommand{\hcube}{H}

\newcommand{\tbranch}{\mathrm{B}}

\newcommand{\ninnin}[1]{\mathcal{I}_{2}(#1)}

\newcommand{\treesettwo}{\mathscr{PT}}

\DeclareMathOperator{\diam}{diam}

\newcommand{\soeji}{\ell}
\makeatletter
\@addtoreset{equation}{section}

\makeatother

\begin{document}

\title[Trees]
{
Metric trees
in 
 the Gromov--Hausdorff space
 }

\author[Yoshito Ishiki]
{Yoshito Ishiki}
\address[Yoshito Ishiki]
{\endgraf
Photonics Control Technology Team
\endgraf
RIKEN Center for Advanced Photonics
\endgraf
2-1 Hirasawa, Wako, Saitama 351-0198, Japan}
\email{yoshito.ishiki@riken.jp}


\maketitle

\begin{abstract}
Using the wedge sum of metric spaces, 
for all compact metrizable spaces, 
we  construct a topological embedding of 
the compact metrizable space into
the set of all 
metric trees in 
 the Gromov--Hausdorff space with finite  prescribed values. 
As its application, we show that 
the set of all metric trees is path-connected 
and 
its all non-empty open subsets 
have infinite topological dimension. 
\end{abstract}

\section{Introduction}\label{sec:intro}
In \cite{Ishiki2021branching},
by constructing continuum many geodesics in the Gromov--Hausdorff space, parametrized by a Hilbert cube, 
the author proved that 
sets of all spaces satisfying  some of  
the doubling property, the uniform disconnectedness, 
the uniform perfectness, 
and   
 sets of 
all infinite-dimensional spaces, 
and the set of all metric spaces homeomorphic to the Cantor set have infinite topological dimension. 

In \cite{Ishiki2021fractal}, 
by constructing topological embeddings of 
compact metrizable spaces
into the Gromov--Hausdorff space, 
the author  proved 
that 
the set of all compact metrizable spaces 
possessing 
prescribed  
topological dimension,  
Hausdorff dimension, 
packing dimension, 
upper box dimension, 
and Assouad dimension, 
and the set of all compact  ultrametric spaces
are  path-connected and
 have  infinite topological dimension. 
The proof is based on 
  the direct sum of metric spaces. 

In \cite{ishiki2021continua}, 
by a similar method to \cite{Ishiki2021fractal} 
(constructing a topological embedding of compact metrizable spaces), 
the author proved that 
each of the sets of all connected, path-connected, geodesic, and CAT(0) compact metric spaces is path-connected and their all non-empty open subsets have infinite topological dimension
in the Gromov--Hausdorff space. 
The proof  is based on the $\ell^{2}$-product metric of  the direct product  of metric spaces. 

As a related work to these author's papers \cite{Ishiki2021branching}, \cite{Ishiki2021fractal} and 
\cite{ishiki2021continua},  
in the present  paper, 
we prove that 
the set of all metric trees is path-connected and 
its all  non-empty open subsets have infinite topological dimension in the   Gromov--Hausdorff spaces. 
In contrast to \cite{Ishiki2021fractal} and 
\cite{ishiki2021continua}, 
we use  the  wedge sum of metric spaces in the proof. 

Let $(X, d)$ be a metric space. Let $x, y\in X$. 
A subset $S$ of a metric space is said to be a 
\emph{geodesic segment connecting $x$ and $y$} if 
there exist a closed interval $[a, b]$ of $\rr$ and 
an isometric embedding $f: [a, b]\to X$ such that 
$f(a)=x$, $f(b)=y$, and $S=f(X)$. 
A metric space is said to be a 
\emph{geodesic space} if 
for all two points, 
there exists a geodesic segment connecting  them. 
A metric space $(X, d)$ is said to be a \emph{metric tree} or 
\emph{$\rr$-tree} if it is 
a geodesic space and if geodesic segments 
$G_{1}$ and $G_{2}$ 
connecting $x, y$ and $y, z$ with $G_{1}\cap G_{2}=\{y\}$ satisfies that $G_{1}\cup G_{2}$ is a geodesic segment 
 connecting $x$ and $z$ for all distinct $x, y, z\in X$ (see \cite{evans2008probability}). 
For a metric space 
$(Z, h)$,  
and  for subsets 
$A$,  $B$ 
of 
$Z$, 
we denote by 
$\hdis(A, B; Z, h)$ 
the \emph{Hausdorff distance of 
$A$ 
and
$B$ 
in 
$(Z, h)$}. 
For  metric spaces 
$(X, d)$ 
and 
$(Y, e)$, 
the 
\emph{Gromov--Hausdorff distance} 
$\grdis((X, d),(Y, e))$ between 
$(X, d)$ 
and 
$(Y, e)$ 
is defined as 
the infimum of  all values  
$\hdis(i(X), j(Y); Z, h)$, 
where 
$(Z, h)$ 
is a metric space, 
 and 
$i: X\to Z$ 
and 
$j: Y\to Z$ 
are isometric embeddings. 
We denote by 
 $\grsp$  
 the  set of all isometry classes of
  non-empty compact metric spaces, and 
 denote by 
 $\grdis$
the Gromov--Hausdorff distance. 
The space 
$(\grsp, \grdis)$ is called   the 
\emph{Gromov--Hausdorff space}.
By abuse of notation, 
we represent an element of $\grsp$ as 
 a pair $(X, d)$ of a set $X$ and a metric $d$ rather than its isometry class. 
We denote by $\treeset$ the set of all metric trees in 
$\grsp$. 
Our  main result is the following theorem, which 
is an analogue of 
\cite[Theorem 1.3]{Ishiki2021fractal} and 
\cite[Theorem 1.1]{ishiki2021continua} for metric trees. 
\begin{thm}\label{thm:treeemb}
Let $n\in \zz_{\ge 1}$. 
Let $\{(X_{i}, d_{i})\}_{i=1}^{n+1}$ be a 
sequence  in $\treeset$ such that 
$\grdis((X_{i}, d_{i}), (X_{j}, d_{j}))>0$ for all  distinct $i, j$.
Let $\hcube$ be a compact metric space and 
$\{v_{i}\}_{i=1}^{n+1}$ be $n+1$ different points 
 in $\hcube$. 
Then, 
there exists a topological embedding 
$\Phi: \hcube\to \treeset$ such that 
$\Phi(v_{i})=(X_{i}, d_{i})$. 
\end{thm}

Applying Theorem \ref{thm:treeemb} to 
$H=[0, 1]^{\aleph_{0}}$, 
we obtain:
\begin{cor}\label{cor:treecor}
The set $\treeset$
is   path-connected and  its  all  non-empty open 
subsets  have infinite topological dimension. 
\end{cor}
We can also obtain 
an analogue of 
Theorem \ref{thm:treeemb} for rooted (pointed) proper metric trees 
(see Subsection \ref{subsec:adrmk}). 
Since it can be proven by 
 the same method of Theorem \ref{thm:treeemb},
 we omit the proof. 

\begin{ac}
The author would like to thank Takumi Yokota  
for 
raising  questions,  
for the many stimulating conversations, and 
for the many  helpful comments. 
\end{ac}


\section{Proof of Theorem}\label{sec:pot}

\subsection{Metric trees}
To prove our results, 
we first discuss the basic properties of metric trees. 
A metric space $(X, d)$ is said to be 
 \emph{$0$-hyperbolic} or 
 satisfy the \emph{four point condition} if 
 \[
 d(x, y)+d(z, t)
 \le 
 \max\lbran 
 d(x, z)+d(y, t), \, 
 d(y, z)+d(x, t)
 \rbran
 \]
 for all $x, y, z, t\in X$. 
The next is proven in 
\cite[Theorem 3.40]{evans2008probability}. 
\begin{prop}\label{prop:0hyp}
A metric space is a metric tree if and only if 
it is connected and $0$-hyperbolic. 
\end{prop}

All metric trees are uniquely geodesic, i.e., 
for each pair of points, 
there uniquely exists a geodesic segment connecting 
the two points (see \cite[Lemmas 3.5 and 3.20]{evans2008probability}). 
Let $(X, d)$ be a metric tree. 
Based on the fact mentioned above, 
for $x, y\in X$, 
we denote by  $[x, y]$ the geodesic segment  connecting $x$ and $y$. 
We also put $[x, y]^{\zantei}=[x, y]\setminus \{x, y\}$. 

Since all metric subspace of a $0$-hyperbolic space 
is $0$-hyperbolic, 
by Proposition 
\ref{prop:0hyp}, 
we obtain:
\begin{lem}\label{lem:pathcontain}
A connected subset $S$ of a metric tree is 
a metric tree itself. 
In particular, for all $x, y\in S$ we have 
$[x, y]\subset S$. 
\end{lem}
 
The next is proven in 
\cite[Lemma 3.20]{evans2008probability}. 
\begin{lem}\label{lem:ylem}
Let $(X, d)$ be a metric tree. 
For all $o, x, y\in X$, 
there exists a unique $q\in X$ such that 
$[o, x]\cap [o, y]=[o, q]$. 
\end{lem}

Let $X$ be a topological space and $x\in X$. 
We denote by 
$\mea(x; X)$ the cardinality  of 
connected components of $X\setminus \{x\}$. 
We put
$\sansan{X}=\lbran\, 
x\in X\mid 
\mea(x; X)\ge 3
\, \rbran$, and put
$\ninnin{X}=\lbran\, 
x\in X\mid 
\mea(x; X)\le 2
\, \rbran$. 
Note that $\ninnin{X}=X\setminus \sansan{X}$, 
and note that 
$\sansan{X}$ and $\ninnin{X}$ are invariant under 
homeomorphisms. 

\begin{lem}\label{lem:hikaku}
Let $(X, d)$ be a metric tree.  
Let $C$ be  a connected component  of 
$\ninnin{X}$. 
Let $o, x, y\in C$. 
Then,  we have  $[o, x]\cap [o, y]=\{o\}$, or 
 $[o, x]\subset [o, y]$, or 
 $[o, y]\subset [o, x]$. 
\end{lem}
\begin{proof}
It suffices to show that 
 the negation of the first conclusion (
 $[o, x]\cap [o, y]\neq\{o\}$) implies 
 either of  the other conclusions. 
By Lemma \ref{lem:ylem}, 
there exists $q\in X$ such that 
$[o, x]\cap [o, y]=[o, q]$. 
By   $[o, x]\cap [o, y]\neq \{o\}$, 
we have
  $q\neq o$. 
Suppose that  $q\neq x$ and $q\neq y$. 
Then we obtain $\mea\lbra q; X\rbra\ge 3$. 
Lemma \ref{lem:pathcontain} implies that 
$q\in C$, and hence $q\in \ninnin{X}$. 
This is a contradiction. 
Thus $q=x$ or $q=y$, which leads to the lemma. 
\end{proof}
\begin{prop}\label{prop:Ciso}
Let $(X, d)$ be a metric tree.  
If a connected component $C$ of 
$\ninnin{X}$ contains at least two points, 
then 
$C$ is 
isometric  to an interval of $\rr$. 
\end{prop}
\begin{proof}
Since $C$ is connected, 
we only need to show the existence of an isometric
embedding of $C$ into $\rr$. 
Take points $o, a, b\in C$ such that 
$o\in [a, b]^{\zantei}\subset  C$. 
We define a map $f: C\to \rr$ by 
\[
f(x)=
\begin{cases}
d(o, x) \  \text{if $b\in [o, x]$ or $x\in [o, b]$;}\\
-d(o, x)\ \text{if $a\in [o, x]$ or $x\in [o, a]$.}
\end{cases}
\]
By Lemma \ref{lem:hikaku}, and by $o\not\in \sansan{X}$, 
the map $f$ is well-defined. 
By the definitions of $f$ and metric trees,  the map 
$f$ is an isometric embedding. 
This finishes the proof. 
\end{proof}

\begin{cor}\label{cor:connectedpiso}
Let $(X, d)$ be a metric tree. 
Then, 
every connected component $C$ of 
$\ninnin{X}$ is isometric to 
either a singleton or 
 an (non-degenerate) interval of $\rr$. 
\end{cor}

For a metric space $(X, d)$ and a subset $A$, 
we denote by $\diam_{d}(A)$ the diameter of $A$. 
\begin{cor}\label{cor:decom}
Let $(X, d)$ be a  metric tree. 
Then, there exist a set $I$ and  points 
$\{a_{\soeji}\}_{\soeji\in I}$ and 
$\{b_{\soeji}\}_{\soeji\in I}$ in $X$ such that 
$\ninnin{X}=
\bigcup_{\soeji\in I}[a_{\soeji}, b_{\soeji}]$
and 
the set 
$[a_{\soeji}, b_{\soeji}]\cap [a_{\soeji'}, b_{\soeji'}]$ contains only   at most one point for all 
distinct $\soeji, \soeji'\in I$, 
and $\diam_{d}([a_{\soeji}, b_{\soeji}])\le 1$ 
for all $\soeji\in I$. 
\end{cor}
\begin{proof}
Since 
every interval of $\rr$ can be represented as the 
union of an at most countable  family of closed intervals with diameter $\le 1$  such that the intersection of  each pair of different members in  the family contains only at most one point, 
we obtain the corollary
by Proposition \ref{prop:Ciso}.  
\end{proof}

\begin{rmk}
There exists a metric tree $(X, d)$ such that the set 
$\sansan{X}$ is dense in $X$. 
By recursively applying  Proposition 
\ref{prop:replacetreeconti} to 
the metric tree $[0, 1]$, we obtain such a tree. 
Thus, in Corollary \ref{cor:decom}, it can happen that 
the index set $I$ is   empty. 
\end{rmk}

\subsection{Specific metric trees}
To show the existence a topological embedding stated in 
Theorem \ref{thm:treeemb}, we construct specific metric trees. 

\begin{df}
We put
$\intvl=[0, 1]$. 
We  construct a family of comb-shaped metric trees parametrized by $\intvl$.
In what follows, we fix a sequence $\{c_{n}:\rr\to \rr\}_{n\in \zz_{\ge 0}}$ of 
 continuous functions such that 
for each $n\in \zz_{\ge 0}$, 
 we have 
$c_{n}(s)=0$
for all $s\in [2^{-n}, \infty)$, 
 and 
$c_{n}(s)\in (0,  1]$ for all $s\in [-\infty, 2^{-n})$. 
To simplify our description, 
we represent  an element $(x, s)$ of 
$\intvl\times \intvl$ as 
$x_{s}$. For example, $0_{0}=(0, 0)$, and 
$(1/3)_{1/2}=(1/3, 1/2)$. 
Let  $w$ denote the metric on $\intvl \times \intvl$ 
defined by $w(x_{s}, y_{t})=s+|x-y|+t$. 
Then, the space $(\intvl \times \intvl, w)$ become  a metric tree. 
For each $n\in \zz_{\ge 0}$
Put 
$I_{n}=
\lbran\, 
m\cdot 2^{-(n+1)}
\mid 
m\in 
\lbran 
0, \dots, 2^{n+1}
\rbran
\, 
\rbran$. 
We also put 
$J_{0}=I_{0}$ and 
$J_{n+1}=I_{n+1}\setminus I_{n}$ for 
$n\in \zz_{\ge 0}$. 
For each $s\in \intvl$, 
we define a subset $\tbranch(s)$ of $\intvl\times \intvl$ by 
\[
\tbranch(s)=
\mathbb{I}\times \lbran0\rbran\cup 
\bigcup_{n\in \zz_{\ge 0}}\bigcup_{a\in J_{n}}
\{a\}\times [0, s\cdot c_{n}(s)]. 
\]
Let 
$w[s]=w|_{\tbranch(s)^{2}}$. 
Then $(\tbranch(s), w[s])$ is a compact metric tree 
for all $s\in \intvl$. Note  that $(\tbranch(0), w[0])$ is isometric to 
$\intvl$. 
\end{df}
By the definition of $\tbranch(s)$, we obtain the next two lemmas. 
\begin{lem}\label{lem:treeconti}
Let  $s\in [0, 1)$. Then the following hold true. 
\begin{enumerate}
\item 
If $s=0$, for all $t\in \intvl$ we have 
$\hdis\lbra\tbranch(t), \tbranch(0); \intvl\times \intvl, w\rbra
\le t$. 
\item 
If $s\neq 0$,  taking  $n\in \zz_{\ge 0}$ with 
$2^{-(n+1)}\le s<2^{-n}$, 
 for all $t\in \intvl$ with $|s-t|<2^{-(n+2)}$,  
we have 
\[
\hdis\lbra\tbranch(t), \tbranch(s); \intvl\times \intvl, w\rbra
\le \max_{0\le i\le n+1}|s\cdot c_{i}(s)-t\cdot c_{i}(t)|. 
\]
\end{enumerate}
\end{lem}

\begin{lem}\label{lem:horizonlength}
Let $n\in \zz_{\ge 0}$ and let $s\in (0, 2^{-n})$. 
Let $C$ be a connected component of 
$\ninnin{\tbranch(s)}$. 
Then we have 
$\diam_{w[s]}(C)<2^{-n}$. 
\end{lem}

A topological  space is said to be a \emph{Hilbert cube} if 
it is homeomorphic to the countable power of 
the closed unit interval $[0, 1]$ of $\rr$.

We now introduce a family of star-shaped  metric trees parametrized by a Hilbert cube (Definition \ref{df:tree}),  which  was  first constructed in  
\cite{ishiki2021continua}. 

We define 
$\qqcube=\prod_{i=1}^{\infty}[2^{-2i}, 2^{-2i+1}]$. 
Note that every $\sqsq{a}=\{a_{i}\}_{i\in \zz_{\ge 1}}\in \qqcube$ satisfies $a_{i}<1$ and $a_{i+1}<a_{i}$ for all $i\in \zz_{\ge 1}$  and $\lim_{i\to \infty}a_{i}=0$. 
We define a metric $\tau$ on $\qqcube$ by 
$\tau(x, y)=\sup_{i\in \zz_{\ge 1}}|x_{i}-y_{i}|$. 
Then,  $\tau$ generates the topology which makes 
$\qqcube$  a Hilbert cube.

\begin{df}\label{df:tree}
Let $\sqsq{a}=\{a_{i}\}_{i\in \zz_{\ge 1}}\in \qqcube$. 
We supplementally put $a_{0}=1$. 
Put 
$\ctree=\{(0, 0)\}\cup (0, 1]\times \zz_{\ge 0}$. 
To simplify our description, 
we represent an element $(s, i)$ of $\ctree$ as 
$s_{i}$. For example, $0_{0}=(0, 0)$, 
$1_{n}=(1, n)$, 
and $(1/2)_3=(1/2, 3)$. 
We define a metric $\disdis[\sqsq{a}]$ on
$\ctree$ by 
\[
\disdis[\sqsq{a}](s_{i}, t_{j})
=
\begin{cases}
a_{i}|s-t| & \text{ if $i=j$ or $st=0$;}\\
a_{i}s+a_{j}t & \text{ otherwiese.}
\end{cases}
\]
Then the space $(\ctree, \disdis[\sqsq{a}])$ is 
a compact metric tree. 
Note that 
even if $\sqsq{a}\neq \sqsq{b}$, 
the metrics 
$\disdis[\sqsq{a}]$ and $\disdis[\sqsq{b}]$ generate 
the same topology on $\ctree$. 
\end{df}

The following propositions are 
\cite[Propositions 2.2 and 2.3]{ishiki2021continua}. 
\begin{prop}\label{prop:qqinj}
Let $\sqsq{a}=\{a_{i}\}_{i\in \zz_{\ge 1}}$ and 
$\sqsq{b}=\{b_{i}\}_{i\in \zz_{\ge 1}}$ be in $\qqcube$. 
Let $K, L\in (0, \infty)$. 
If $(\ctree, K\cdot\disdis[\sqsq{a}])$ and 
$(\ctree, L\cdot \disdis[\sqsq{b}])$
are isometric to each other, 
then $\sqsq{a}=\sqsq{b}$. 
\end{prop}

\begin{prop}\label{prop:qqconti}
For all  $\sqsq{a}, \sqsq{b}\in \qqcube$, 
we obtain
\[
\sup_{x, y\in \ctree}
|\disdis[\sqsq{a}](x, y)-\disdis[\sqsq{b}](x, y)|\le 2\tau(\sqsq{a}, \sqsq{b}).
\] 
\end{prop}


\subsection{Amalgamation of metrics}

The following proposition shows a way of constructing 
the wedge sum of metric spaces. 
The statement (\ref{item:metric1})  is deduced from \cite[Proposition 3.2]{Ishiki2020int}.
The statement (\ref{item:metric3}) follows from \cite[Proposition 2.6]{memoli2021characterization} and 
the definition of metric trees. 
\begin{prop}\label{prop:amal1}
Let $k\in \zz_{\ge 2}$. 
Let $\{(X_{i}, d_{i})\}_{i=1}^{k}$ be a sequence of  metric spaces. 
Assume that there exists a point $p$ such that  
$X_{i}\cap X_{j}=\{p\}$ for all distinct $i, j\in \{1, \dots, k\}$. 
We define 
a symmetric function 
$h: \lbra\bigcup_{i=1}^{k}X_{i}\rbra^2\to [0, \infty)$ 
by
\begin{align*}
	h(x, y)=
		\begin{cases}
		d_{i}(x, y) & \text{if $x, y\in X_{i}$;}\\
		d_{i}(x, p)+d_{j}(p, y) & \text{if $(x, y)\in X_{i}\times X_{j}$ and $i\neq j$. }
		\end{cases}
\end{align*}
Then, the following statements hold  true. 
\begin{enumerate}
\item 
The function $h$ is a metric with
$h|_{X_{i}^2}=d_{i}$ for all $i\in \{1, \dots, k\}$. \label{item:metric1}
\item 
If each $(X_{i}, d_{i})$ is a
geodesic  metric space (resp.~metric tree), 
then so is $\lbra \bigcup_{i=1}^{k}X_{i}, h\rbra$. \label{item:metric3}
\end{enumerate}
\end{prop}

To prove our theorem, we need an
operation of replacing  edges of a metric tree 
by 
other 
metric trees. 
\begin{df}\label{df:replace}
Let $(X, d)$ be a metric tree, and 
$\{a_{\soeji}\}_{\soeji\in I}$ 
and 
$\{b_{\soeji}\}_{\soeji\in I}$ be 
families  of points in $X$ such that 
$[a_{\soeji}, b_{\soeji}]\cap 
[a_{\soeji'}, b_{\soeji'}]$ 
contains only at most one point for all distinct 
$\soeji, \soeji'\in I$. 
Let $\{(T_{\soeji}, e_{\soeji}, \alpha_{\soeji}, \beta_{\soeji})\}_{\soeji\in I}$ be a family of 
quadruple of metric trees 
$(T_{\soeji}, e_{\soeji})$ and  two specified points 
$\alpha_{\soeji}, \beta_{\soeji}\in T_{\soeji}$ such that 
$e_{\soeji}(\alpha_{\soeji}, \beta_{\soeji})=d(a_{\soeji}, b_{\soeji})$. 
Now 
we remove the sets $[a_{\soeji}, b_{\soeji}]^{\zantei}$ from $X$,  and 
identify 
$a_{\soeji}, b_{\soeji}$ 
with $\alpha_{\soeji}, \beta_{\soeji}$, respectively,  and 
consider that 
$X\cap T_{\soeji}=\{a_{\soeji}, b_{\soeji}\}$. 
Let $Y$ denote the resulting set. 
For $x\in Y$, 
let $E(x)=\{a_{\soeji}, b_{\soeji}\}$ and $h_{x}=e_{\soeji}$ if 
$x\in [a_{\soeji}, b_{\soeji}]$; otherwise, $E(x)=\{x\}$ and $h_{x}=d$. 
For each $x, y \in Y$, 
we define $u_{(x, y)}\in E(x)$ and $v_{(x, y)}\in E(y)$ by the points such that $d\lbra u_{(x, y)}, v_{(x, y)}\rbra$ is equal to  the distance between 
the sets 
$E(x)$ and $E(y)$. 
Note that the points $u_{(x, y)}$ and $v_{(x, y)}$ 
uniquely exist and $u_{(x, y)}=v_{(y, x)}$ 
and $v_{(x, y)}=u_{(y, x)}$. 
We define a symmetric function $D$ on $Y^{2}$ by 
$D(x, y)=h_{x}(x, u_{(x, y)})+d(u_{(x, y)}, v_{(x, y)})+
h_{y}(v_{(x, y)}, y)$. Then $D$ is a metric and the space 
$(Y, D)$ is a metric tree. 
We call this space a \emph{metric tree induced from $(X, d)$ replaced by 
$\{(T_{\soeji}, e_{\soeji}, \alpha_{\soeji}, \beta_{\soeji})\}_{\soeji\in I}$  with respect to 
$\{a_{\soeji}\}_{\soeji\in I}$ and 
$\{b_{\soeji}\}_{\soeji\in I}$}. 
Note that since $[a_{\soeji}, b_{\soeji}]$ is isometric to 
$[\alpha_{\soeji}, \beta_{\soeji}]$, 
the space $(Y, D)$ contains the original 
metric tree  $(X, d)$ as a metric subspace. 
\end{df}
\begin{prop}\label{prop:replacetreeconti}
Let $(X, d)$ be a metric tree. 
Let $\{a_{\soeji}\}_{\soeji\in I}$ and 
$\{b_{\soeji}\}_{\soeji\in I}$ be 
points stated in Corollary \ref{cor:decom}. 
Put $M_{\soeji}=d(a_{\soeji}, b_{\soeji})$. 
For each $s\in \intvl$, 
let $(Y(s), D[s])$ be the 
metric tree induced from $(X, d)$ replaced by 
$\{(\tbranch(s), M_{\soeji}\cdot w[s], 0_{0}, 1_{0})\}_{\soeji\in I}$ 
with respect to 
$\{a_{\soeji}\}_{\soeji\in I}$ and $\{b_{\soeji}\}_{\soeji\in I}$. 
Then the following statements hold true. 
\begin{enumerate}
\item 
The space $(Y(0), D[0])$ is isometric to $(X, d)$. 
\item 
For all $s\in \intvl$, we have 
$\lim_{t\to s}\grdis((Y(s), D[s]), (Y(t), D[t]))=0$. 
\end{enumerate}
\end{prop}
\begin{proof}
Since $(\tbranch(0), w[0])$ is isometric to $\intvl$, 
the statement (1) holds true. 
The statement (2) follows from Lemma \ref{lem:treeconti} and $M_{\soeji}\le 1$. 
\end{proof}


\subsection{Topological embeddings}
For a metric space $(X, d)$, 
$o\in X$, and $r\in [0, \infty]$,  we denote by 
$B(o, r)$
the set of all $x\in X$ with $d(o, x)\le r$. 
Note that  $B(x, 0)=\{x\}$ and $B(x, \infty)=X$. 
\begin{lem}\label{lem:ballsconti}
Let $(X, d)$ be a geodesic space. 
Let $o\in X$. Then, for all $r, r'\in [0, \infty)$, we have 
$\hdis(B(o, r), B(o, r'); X, d)\le |r-r'|$. 
\end{lem}

For every $n\in \zz_{\ge 1}$, 
 we denote by $\widehat{n}$
the set $\{1, \dots, n\}$. 
In what follows, we consider that 
the set $\widehat{n}$ 
is  equipped with the discrete topology. 

The following proposition has an  essential  role in 
the proof of Theorem \ref{thm:treeemb}. 
Using this proposition, 
Theorem \ref{thm:treeemb} can be proven by 
an elementary argument such as the pigeonhole principle. 
Similar propositions are shown  in 
\cite[Proposition 4.4]{Ishiki2021fractal} 
and 
\cite[Propositions 3.6 and 4.2]{ishiki2021continua}, 
which  proofs  are based on the direct sum and 
direct product of metric spaces, respectively. 
Unlike these propositions, the following is based on 
the wedge sum of metric spaces discussed in Proposition \ref{prop:amal1}. 
\begin{prop}\label{prop:embednm}
Let $n\in \zz_{\ge 1}$ and $m\in \zz_{\ge 2}$. 
Let $H$ be a compact metrizable spaces, 
and $\{v_{i}\}_{i=1}^{n+1}$ be 
$n+1$ different points in $H$. 
Put $\deltasecond=H\setminus \{\, v_{i}\mid i=1, \dots, n+1\, \}$. 
Let $\{(X_{i}, d_{i})\}_{i=1}^{n+1}$ be a 
sequence of compact metric spaces in $\treeset$ satisfying  that 
$\grdis((X_{i}, d_{i}), (X_{j}, d_{j}))>0$ for all distinct $i, j$. 
Then there exists a continuous map 
$F: \hcube \times \widehat{m}\to 
\treeset$
such that 
\begin{enumerate}
\item 
for all $i\in \widehat{n+1}$ and $k\in \widehat{m}$ we have 
$F(v_{i}, k)=(X_{i}, d_{i})$;
\item 
for all $(u, k), (u', k')\in \deltasecond\times \widehat{m}$ with $(u, k)\neq (u', k')$, we have 
$F(u, k)\neq F(u', k')$. 
\end{enumerate}
\end{prop}
\begin{proof}
In what follows, 
we consider that the set $[0, \infty]$ is equipped with 
the canonical  topology homeomorphic to $[0, 1]$. 
Since every metrizable space is perfectly normal, 
and since $[0, \infty]$ is homeomorphic to $[0, 1]$, 
for each $i\in \widehat{n+1}$
we can  take a continuous function 
$\sigma_{i}:H\to [0, \infty]$ such that 
$\sigma_{i}^{-1}(0)=\{\, v_{j}\mid j\neq i\, \}$ and
$\sigma_{i}^{-1}(\infty)=\{v_{i}\}$. 
We can also take a continuous function 
$\varphi:H\to [0, 1/2]$ with 
$\varphi^{-1}(0)=\lbran\, v_{i}\mid i=1, \dots, n+1\,  \rbran$. 
We put $\xi(u)=32\cdot \varphi(u)$. 
Since $H\times \widehat{m}$ is compact and metrizable, 
there exists a topological embedding 
$\topemb:H\times  \widehat{m}\to \qqcube$ 
(this is the Urysohn metrization  theorem, see \cite{Kelly1975}).

For each $i\in \widehat{n+1}$, 
let $\{a_{i, \soeji}\}_{\soeji\in I}$ and $\{b_{i, \soeji}\}_{\soeji\in I}$ 
be points in $(X_{i}, d_{i})$ stated in Corollary \ref{cor:decom}. 
Put $M_{i, \soeji}=d(a_{i, \soeji}, b_{i, \soeji})$. 
Then,  we have $M_{i, \soeji}\le 1$.
For each $s\in \intvl$, 
we denote by $(Y_{i}(s), D_{i}[s])$ the 
 metric tree induced from $(X_{i}, d_{i})$ replaced by 
$\{(\tbranch(s), M_{i, \soeji}\cdot w[s], 0_{0}, 1_{0})\}_{\soeji\in I}$ 
 with respect to 
$\{a_{i, \soeji}\}_{\soeji\in I}$ and $\{b_{i, \soeji}\}_{\soeji\in I}$. 

For each $i\in \widehat{n+1}$,  
we take $p_{i}\in X_{i}$. 
For each $(u, k)\in H\times \widehat{m}$, 
we denote by 
$Z_{i}(u, k)$ the set of all $x\in Y_{i}(\varphi(u))$
with $D_{i}[\varphi(u)](x, p_{i})\le \sigma_{i}(u)$. 
Let $E_{i}[u, k]$ denote the 
restricted metric of $D_{i}[\varphi(u)]$ on $Z_{i}(u, k)$. 
Put 
$(Z_{n+2}(u, k), E_{n+2}[u, k])=
(\ctree, \xi(u)\cdot \disdis[\rho(u, k)])$ and 
$p_{n+2}=1_{0}\in \ctree$. 

We identify the $n+2$ many points 
$\{\, p_{i}\mid i=1, \dots n+2\, \}$
as a single point, say $p$, 
and we consider that 
$Z_{i}(u, k)\cap Z_{i'}(u, k)=\{p\}$ for all distinct 
$i, i'\in \widehat{n+1}$. 
We put 
$W(u, k)=\bigcup_{i=1}^{n+2}Z_{i}(u, k)$.
Applying  Proposition  \ref{prop:amal1}, we obtain a
metric 
$g[u, k]$ on $W(u, k)$ such that 
$g[u, k]|_{Z_{i}(u, k)^{2}}=E_{i}[u, k]$.  
Namely, the space $(W(u, k), g[u, k])$ is the 
wedge sum of the spaces $\{(Z_{i}(u, k), E[u, k])\}_{i=1}^{n+2}$ with respect to the points $\{\, p_{i}\mid i=1, \dots n+2\, \}$. 

By (2) in Proposition  \ref{prop:amal1}, we see that 
$(W(u, k), g[u, k])$ 
is a metric tree for all $(u, k)\in H\times \widehat{m}$.
By (1) in Proposition \ref{prop:replacetreeconti}, 
note  that  $(W(v_{i}, k), g[v_{i}, k])$  is isometric to 
$(X_{i}, d_{i})$ for all $i\in \widehat{n+1}$ and 
$k\in \widehat{m}$. 
We define  
$F: H\times \widehat{m}\to \treeset$ 
 by 
\[
F(u, k)=
\begin{cases}
(X_{i}, d_{i}) & \text{if $u=v_{i}$ for some $i\in \widehat{n+1}$;}\\
\left(W(u, k), g[u, k]\right) & \text{otherwise.}
\end{cases}
\]
By (2) in Proposition \ref{prop:replacetreeconti},  and
Proposition \ref{prop:qqconti} and 
Lemma \ref{lem:ballsconti}, and the continuity of each $\sigma_{i}$,  
the map  $F$ is continuous. 
By the definition, the condition (1) is satisfied.

To prove the condition (2) in the proposition, 
we assume that there exists an isometry 
$f: (W(u, k), g[u, k])\to (W(u', k'), g[u', k'])$. 

We first  show that $f(\ctree)=\ctree$. 
Fix arbitrary $(v, l)\in \deltasecond\times \widehat{m}$. 
Let $\mathcal{P}(v, l)$ be the set of all connected components of 
$\ninnin{W(v, l)}$. 
Take $a\in \zz_{\ge 0}$ with 
$2^{-(a+1)}\le \varphi(v)<2^{-a}$. 
Let $C$ be a connected component of 
$\ninnin{\bigcup_{i=1}^{n+1}Z_{i}(v, l)}$. 
Then by Lemma \ref{lem:horizonlength}, and 
by $M_{i, \soeji}\le 1$,  
 we have 
$\diam_{g[v, l]}(C)<2^{-a}$. 
Since $2^{-a}\le 2\varphi(v)=2^{-4}\xi(v)$, 
we obtain $\diam_{g[v, l]}(C)<2^{-4}\xi(v)$. 
Since $2^{-4}\le \disdis[\sqsq{a}](0_{0}, 1_{1})$ for all 
$\sqsq{a}\in \qqcube$, we have 
$2^{-4}\xi(v)\le g[v, l](0_{0}, 1_{1})$. 
By the definitions of $\ctree$ and $g$, 
we have 
$g[v, l](0_{0}, 1_{i+1})<g[v, l](0_{0}, 1_{i})$ for all $i\in \zz_{\ge 0}$. 
Thus, 
we conclude that 
the subset 
$[0_{0}, 1_{0}]^{\zantei}\cup\{1_{0}\}$ of $\ctree$ is the unique set possessing the maximal 
diameter of  elements in 
$\mathcal{P}(v, l)$, 
and the subset $[0_{0}, 1_{1}]^{\zantei}\cup\{1_{1}\}$ of 
$\ctree$ is the unique set possessing 
the second maximal diameter of elements in 
$\mathcal{P}(v,  l)$. 
Putting $(v, l)=(u, k), (u', k')$, 
since $f$ is an isometry, 
by the argument discussed above, 
we obtain
$f([0_{0}, 1_{0}])=[0_{0}, 1_{0}]$ 
and 
$f([0_{0}, 1_{1}])=
[0_{0}, 1_{1}]$. 
This implies that 
$f(0_{0})\in \{0_{0}, 1_{0}\}$ and 
$f(0_{0})\in \{0_{0}, 1_{1}\}$. 
Thus $f(0_{0})=0_{0}$,  and $f(1_{i})=1_{i}$
for all $i\in \{0, 1\}$. 

To prove $f(\ctree)=\ctree$, 
for the sake of contradiction, 
we suppose that 
there exists $x\in \ctree$ with 
$f(x)\not\in \ctree$. Take $q\in \zz_{\ge 0}$
such that $x\in [0_{0}, 1_{q}]$. 
Then, by the construction of $W(u, k)$, 
the segment $[0_{0}, f(x)]$ must contain $1_{0}$. 
Thus, 
$g[u, k](0_{0}, 1_{0})\le 
g[u, k](0_{0}, x)\le g[u, k](0_{0}, 1_{q})$. 
Since 
$g[u, k](0_{0}, 1_{i})<g[u, k](0_{0}, 1_{0})$ for all $i\neq 0$, we obtain $1_{q}=1_{0}=x$. 
This contradicts $f(1_{0})=1_{0}$. 
Therefore $f(\ctree)\subset \ctree$. 
By replacing the role of $f$ with $f^{-1}$, 
we conclude that $f(\ctree)=\ctree$. 

We now prove the condition (2). 
By the definition of $g$, and $f(\ctree)=\ctree$, 
the spaces $(\ctree, \xi(u)\cdot \disdis[\rho(u, k)])$ and 
$(\ctree, \xi(u')\cdot \disdis[\rho(u', k')])$ are isometric to each other. Then, 
by Proposition \ref{prop:qqinj}, 
we have  $\rho(u, k)=\rho(u', k')$, 
and hence $u=u'$ and $k=k'$. 
Therefore we obtain the condition (2). 
This finishes the proof of 
Proposition \ref{prop:embednm}. 
\end{proof}

\begin{proof}[Proof of Theorem \ref{thm:treeemb}]
The proof of Theorem \ref{thm:treeemb} is essentially the 
same  as \cite[Theorem 1.1]{ishiki2021continua}   
and 
\cite[Theorem 1.3]{Ishiki2021fractal}. 
Put $m=n+2$. 
Let 
$F: H\times \widehat{m}\to \treeset$ be a map stated in 
Proposition \ref{prop:embednm}. 
For the sake of contradiction, 
we suppose that 
for all  $k\in \widehat{m}$  we have 
$\{\, (X_{i}, d_{i})\mid i=1, \dots, n+1\, \}\cap 
F\left(\deltasecond\times \{k\}\right)
\neq 
\emptyset$. 
Then, by $m=n+2$, and by the pigeonhole principle, 
there exists two distinct $j, j'\in \widehat{m}$ such that 
$(X_{i}, d_{i})\in F\left(\deltasecond\times \{j\}\right)\cap F\left(\deltasecond\times \{j'\}\right)$ for some 
$i\in \widehat{n+1}$. 
This contradicts 
 the condition (2) in the Proposition \ref{prop:embednm}. 
 Thus,  there exists $k\in \widehat{m}$ such that 
$\{\, (X_{i}, d_{i})\mid i=1, \dots, n+1\, \}\cap 
F\left(\deltasecond\times \{k\}\right)
=\emptyset$. 
Therefore,  the function 
$\Phi:H\to \treeset$ defined by $\Phi(u)=F(u, k)$ is injective, and hence $\Phi$ is  a  topological embedding since $H$ is compact. 
This completes the proof of Theorem \ref{thm:treeemb}. 
\end{proof}
\subsection{Additional remark}\label{subsec:adrmk}
We denote by 
$\grsptwo$ the set of all proper metric spaces equipped with the pointed Gromov--Hausdorff distance
$\pgrdis$ 
(for the definition, see \cite{herron2016gromov} or \cite{ishiki2021continua}). 
Let $\treesettwo$ denote the set of all metric trees in  $\grsptwo$. 
By the same method as the proof of Theorem \ref{thm:treeemb},  
using \cite[Lemma 3.4]{herron2016gromov}
we obtain an analogue of Theorem \ref{thm:treeemb} for proper metric trees. 
We omit the proof of  the following. 
A similar theorem is proven in 
\cite[Theorem 1.3]{ishiki2021continua}, 
and we refer the readers to the proofs of 
\cite[Theorem 1.3]{ishiki2021continua} 
and Theorem \ref{thm:treeemb} in the present paper. 

\begin{thm}\label{thm:ptreeemb}
Let $n\in \zz_{\ge 1}$. 
Let $H$ be a compact metrizable space, 
and 
$\{v_{i}\}_{i=1}^{n+1}$ be $n+1$ different points  in 
$H$. 
Let $\{(X_{i}, d_{i}, a_{i})\}_{i=1}^{n+1}$ be a 
sequence  in $\treesettwo$ such that 
$\pgrdis((X_{i}, d_{i}, a_{i}), (X_{j}, d_{j}, a_{j}))>0$ for all  distinct $i, j$. 
Then, 
there exists a topological embedding 
$\Phi: H\to \treesettwo$ such that 
$\Phi(v_{i})=(X_{i}, d_{i}, a_{i})$. 
\end{thm}

\begin{cor}\label{cor:pontedtree}
The set $\treesettwo$
is   path-connected and  its  all  non-empty open 
subsets  have infinite topological dimension. 
\end{cor}

\bibliographystyle{amsplain}
\bibliography{bibtex/tree.bib}

\end{document}